\documentclass[11pt]{article}

\usepackage[T2A,T1]{fontenc}
\usepackage[utf8]{inputenc}
\usepackage{substitutefont}
\substitutefont{T2A}{\familydefault}{Tempora-TLF}
\usepackage{amsmath}
\usepackage[english,russian]{babel}
\usepackage{amssymb}
\usepackage{amsthm}
\usepackage{paralist}
\usepackage[paper=a4paper,left=25mm,right=25mm,top=25mm,bottom=25mm]{geometry}
\usepackage{hyperref}
\usepackage{xcolor} 
\usepackage{fancyhdr}
\usepackage{thmtools}
\usepackage{tikz-cd}
\usepackage{tikz}
\usetikzlibrary{patterns,snakes}
\tikzstyle{spring}=[line width=0.8,black,snake=coil,segment amplitude=5,segment length=5,line cap=round]
\tikzstyle{mass}=[line width=0.6,black,rounded corners=0]
\tikzstyle{ground}=[fill,pattern=north east lines,draw=none,minimum width=0.3,minimum height=0.6]

\declaretheoremstyle[headfont=\normalfont\bfseries]{bfthmstyle}

\declaretheorem[sharenumber=Theorem,style=bfthmstyle]{Lemma}

\declaretheoremstyle[headfont=\normalfont\bfseries,qed={$\diamond$}]{bfdefstyle}
\declaretheorem[sharenumber=Theorem,style=bfdefstyle]{Definition}
\declaretheorem[sharenumber=Theorem,style=bfdefstyle]{Definition-Proposition}
\declaretheorem[sharenumber=Theorem,style=bfdefstyle]{Definition-Lemma}
\declaretheorem[sharenumber=Theorem,style=bfdefstyle]{Example}

\newtheorem{definition-proposition}[Theorem]{Definition-Proposition} 
\newtheorem{definition-lemma}[Theorem]{Definition-Lemma} 

\declaretheoremstyle[headfont=\normalfont\bfseries,qed={$\diamond$}]{bfremstyle}
\declaretheorem[sharenumber=Theorem,style=bfdefstyle]{Remark}

\newcommand{\abs}[1]{\left\vert #1 \right\vert}
\newcommand{\set}[1]{\left\lbrace #1 \right\rbrace}

\declaretheoremstyle[
notefont=\normalfont, notebraces={}{},
headformat=\mathbb{N}UMBER~\mathbb{N}AME~\mathbb{N}OTE
]{nopar}

\numberwithin{equation}{section}

\author{Jonas Kirchhoff}
\title{A behavioural approach to port-controlled systems}
\date{03.06.2024}

\begin{document}
\setlength{\parindent}{0em}
\pagestyle{fancy}
\lhead{Jonas Kirchhoff}
\rhead{behavioural approach}

\maketitle

\paragraph{Abstract} We give insight in the structure of port-Hamiltonian systems as control systems in between two closed Hamiltonian systems. Using the language of category theory, we identify systems with their behavioural representation and view a port-control structure with desired structural properties on a given closed system as an extension of this system which itself may be embedded in a ``larger'' closed system. The latter system describes the nature of the ports (e.g. Hamiltonian, metriplectic etc.). This point of view allows us to describe meaningful port-control structures for a large family of systems, which is illustrated with Hamiltonian and metriplectic systems.\vspace{-1mm}
 
\vspace{-1mm}

\vfill
\par\noindent\rule{5cm}{0.4pt}\\
\begin{footnotesize}
Corresponding author: Jonas Kirchhoff\\[1em]
Jonas Kirchhoff\\
Institut für Mathematik, Technische Universität Ilmenau, Weimarer Stra\ss e 25, 98693 Ilmenau, Germany\\
E-mail: jonas.kirchhoff@tu-ilmenau.de\\[1em]

Jonas Kirchhoff thanks the Technische Universität Ilmenau and the Freistaat Thüringen for their financial support as part of the Thüringer Graduiertenförderung.

\end{footnotesize}

\newpage

\section{Introduction}

Port-controlled systems, in particular port-Hamiltonian systems, have been extensively studied since their introduction more than thirty years ago. This vast theory covers finite-dimensional electro-mechanical systems~\cite{SchaftJeltsema14,Mehrmann_Unger_2023}, infinite-dimensional systems~\cite{JacoZwar12} and has been extended to thermodynamic systems, see e.g.~\cite{RamiMascSbar13,Ramirez_Entropy22,MascScha19}, However, the assignement of a port-structure does appear to be rather ad-hoc instead of following a well-defined procedure. In the present paper, a behavioural point of view is offered, where port-controlled systems are the behaviours that realise a certain class of choices of extensions of given closed models. 

The paper is organised as follows. In section 2, we recall the behavioural theory of Jan Willems, see~\cite{Polderman}. We demonstrate that instead of following the original definition, the continuous machines of~\cite{SchuSpivVasi20} present more properties of solutions of dynamical systems, in particular differential equations. In section 3, we define port-controlled systems with given interface and demonstrate this with Hamiltonian and metriplectic systems in section 4.

\section{Behaviours and continuous machines}

The original definition of a dynamical system in the behavioural framework is the following.

\begin{Definition}[{\cite[Def. 1.3.1]{Polderman}}]\label{def:Behaviour}
A dynamical system is a triple $\Sigma = (\mathbb{T},\mathbb{W},\mathfrak{B})$ of a \textit{time axis} $\mathbb{T}\subseteq\mathbb{R}$, a \textit{signal space} $\mathbb{W}$ and a \textit{behaviour} $\mathfrak{B}\subseteq\mathbb{W}^\mathbb{T}$ of curves in the signal space.
\end{Definition}

In particular, trajectories of a dynamical system must necessarily be defined on the whole time axis. A particular class of dynamical systems are the dynamical systems which are given by differential equations, where the time axis is a nontrivial interval and a typical choice of behaviour is the set of all maximal solutions. Here, a solution is a differentiable curve, defined on an open interval, which fulfils the differential equation on each point; and a maximal solution is a solution which may not be extended to a larger interval on the time axis. As an example, the set $\mathfrak{B}_0$ of all maximal solutions of
\begin{align}\label{eq:IVP}
\tfrac{\mathrm{d}}{\mathrm{d}t}x = x^2,\qquad x(0) = x_0,
\end{align}
where $x_0$ is allowed to vary arbitrarily in $\mathbb{R}$, contains precisely the functions
\begin{align*}
x(\cdot;x_0):\left(-\frac{1}{\abs{x_0}},\frac{1}{\abs{x_0}}\right)\to\mathbb{R},\quad t\mapsto\frac{1}{\frac{1}{x_0}-t}
\end{align*}
and the constant zero function. Each (not necessarily maximal) solution of~\eqref{eq:IVP} is the restriction of a maximal solution to a smaller interval. In particular, these solutions are not defined on the whole time axis $\mathbb{R}$, and there exists no $\set{0}\subsetneqq\mathbb{T}\subseteq\mathbb{R}$ so that all (not necessarily maximal) solutions of~\eqref{eq:IVP} may be defined on $\mathbb{T}$. Thus, $\mathfrak{B}_0$ is not a behaviour in the strict sense of Definition~\ref{def:Behaviour}. This may be resolved by replacing the time axis by a ``suitable collection'' $(\mathbb{T}_i)_{i\in I}$ of time axes; the behaviour then becomes a family $\mathfrak{B}_i\in\mathbb{W}^{\mathbb{T}_i}$, $i\in I$.

A second property of solutions of differential equations that is not captured by the general Definition~\ref{def:Behaviour}, is that restrictions of (not necessarily maximal) solutions of a differential equation are (in general not maximal) solutions. This means that there exist ``restriction functions'' $\varphi_{i,j}: \mathfrak{B}_i\to\mathfrak{B}_j$ for ``suitable'' time axes $\mathbb{T}_i,\mathbb{T}_j$; in particular, there must be a well-chosen relation between $\mathbb{T}_i,\mathbb{T}_j$. A first choice for time axes in the case of a differential equation is the set of all open, nonempty intervals, with $\mathfrak{B}_I$ being the set of all curves defined on $I$ which solve our differential equation. A quite natural choice for the restriction functions in this case is the the usual restriction of functions to smaller intervals, i.e. $\varphi_{I,J}$ exists if, and only if, $J\subseteq I$, and then $\varphi_{I,J}(x) = x\vert_J$. Then, it becomes clear that a ``suitable collection'' of time axes must be a ``well-chosen'' category, and the behaviour is a presheaf over this category. Recall the formal definition.

\begin{Definition}
A \textit{small category} $\mathcal{C}$ consists of the data
\begin{itemize}
\item a set $\mathrm{Ob}_\mathcal{C}$ of \textit{objects}
\item for each $x,y\in\mathrm{Ob}_\mathcal{C}$ a (possibly empty) set $\mathrm{Hom}_\mathcal{C}(x,y)$ of \textit{morphisms}
\item an operation $\circ: \mathrm{Hom}_\mathcal{C}(y,z)\times\mathrm{Hom}_\mathcal{C}(x,y)\to\mathrm{Hom}_\mathcal{C}(x,z)$
\end{itemize}
so that $\circ$ is associative and admits a unit, i.e. for all $x\in\mathrm{Ob}_\mathcal{C}$ there exists $1_x\in\mathrm{Hom}_\mathcal{C}(x,x)$ so that, for all $y\in\mathrm{Ob}_\mathcal{C}$, $\varphi\in\mathrm{Hom}_\mathcal{C}(x,y)$ and $\psi\in\mathrm{Hom}(y,x)$, $\varphi\circ 1_x = \varphi$ and $1_x\circ\psi = \psi$. The \textit{opposite category} $\mathcal{C}^{\mathrm{op}}$ is the category with $\mathrm{Ob}_{\mathcal{C}^\mathrm{op}} := \mathrm{Ob}_\mathcal{C}$ and $\mathrm{Hom}_{\mathcal{C}^{\mathrm{op}}}(x,y) := \mathrm{Hom}_\mathcal{C}(y,x)$ for all $x,y\in\mathrm{Ob}_\mathcal{C}$. A \textit{functor} $F:\mathcal{C}\to\mathcal{C}'$ between (small) categories $\mathcal{C}$ and $\mathcal{C}'$ consists of the data
\begin{itemize}
\item for all $x\in\mathrm{Ob}_\mathcal{C}$, an object $F(x)\in\mathrm{Ob}_{\mathcal{C}'}$
\item for each $\varphi\in\mathrm{Hom}_\mathcal{C}(x,y)$, a morphism $F(\varphi)\in\mathrm{Hom}_{\mathcal{C}'}(F(x),F(y))$
\end{itemize}
so that $F(\psi\circ\varphi) = F(\psi)\circ F(\varphi)$ and $F(1_x) = 1_{F(x)}$. The functors $F:\mathcal{C}^\mathrm{op}\to\mathcal{C}'$ are called presheaves over $\mathcal{C}$ with values in $\mathcal{C}'$.
\end{Definition}

\begin{Remark}
In particular, the presheaves over a small category $\mathcal{C}$ with values in a category $\mathcal{C}'$ are the objects of the category \textbf{Psh}($\mathcal{C}$,$\mathcal{C}'$), see~\cite[Definition 17.1.3]{KashShap06}. A morphism of presheaves $\mathcal{F}$ and $\mathcal{G}$ consists is a family $\varphi_x\in\mathrm{Hom}_{\mathcal{C}'}(\mathcal{F}(x),\mathcal{G}(x))$, $x\in\mathrm{Ob}_\mathcal{C}$, so that the diagram
\begin{center}
\begin{tikzcd}
\mathcal{F}(x)\arrow{r}{\varphi_x}\arrow{d}{\mathcal{F}(\psi)} & \mathcal{G}(x) \arrow{d}{\mathcal{G}(\psi)}\\
\mathcal{F}(y)\arrow{r}{\varphi_y} & \mathcal{G}(y)
\end{tikzcd}
\end{center}
is commutative for all $x,y\in\mathrm{Ob}_\mathcal{C}$ and each $\psi\in\mathrm{Hom}_\mathcal{C}(y,x)$.
\end{Remark}

\begin{Example}\label{ex:unzureichend}
The open intervals containing zero are the objects of a category $\mathcal{C}$ whose morphisms are the inclusions, i.e. $\mathrm{Hom}_\mathcal{C}(I,J) = \set{\mathrm{id}_I}$ whenever $I\subseteq J$, and otherwise $\mathrm{Hom}_\mathcal{C}(I,J) = \emptyset$. Let $f:\mathbb{R}^{1+n}\to\mathbb{R}^n$ be a continuous function and put, for all $I\in\mathrm{Ob}_\mathcal{C}$,
\begin{align*}
\mathfrak{B}_{f,0}(I) := \set{x\in\mathcal{C}^1(I,\mathbb{R}^n)\,\big\vert\,\forall t\in I: \tfrac{\mathrm{d}}{\mathrm{d}t}x(t) = f(t,x(t))},
\end{align*}
and, for all $I,J\in\mathrm{Ob}_\mathcal{C}$ with $J\subseteq I$ and $x\in\mathfrak{B}_{f,0}(I)$, $\varphi_{I,J}(x) := x\vert_J$. Since differentiability and continuity are by definition local properties, $x\vert_J\in\mathfrak{B}_{f,0}(J)$ and therefore the data $(\mathfrak{B}_{f,0}(I),\varphi_{I,J})_{J\subseteq I\in\mathrm{Ob}_\mathcal{C}}$ are the data of a presheaf with values in the category of sets. This presheaf is the behaviour of the initial value problem
\begin{align*}
\tfrac{\mathrm{d}}{\mathrm{d}t}x = f(t,x),\qquad x(0) = x_0
\end{align*}
with free initial state $x_0$.
\end{Example}

One property of solutions of a differential equation $\tfrac{\mathrm{d}}{\mathrm{d}t}x = f(t,x)$ is that solutions that agree on overlap of their domains may be glued together to a single solution. This is formalised in the notion of a sheaf on a Grothendieck site. Informally\footnote{The formal definition using sieves may be found e.g. in~\cite[chapter 16]{KashShap06}.}, the latter is a category with a choice of covering families of maps corresponding to open coverings in the category of open sets on a topological space; and a sheaf is a presheaf so that families of elements locally defined on a covering family, whose restrictions coincide on common, may be globally extended. Then, a natural choice of coverings on the interval category of Example~\ref{ex:unzureichend} are coverings by increasing families of open subintervals and gluing of solutions is just the choice of the solution defined on the largest element of the covering family. A different category of intervals whose sheaves much better capture the more intricate property that solutions defined on compact intervals $[a,b]$ and $[b,c]$ may be glued together whenever they coincide on $b$, was suggested in~\cite{SchuSpivVasi20}.

\begin{Definition}[{\cite[Definition 3.1.1]{SchuSpivVasi20}}]
The category \textbf{Int} of \textit{continuous intervals} consists of the following data:
\begin{itemize}
\item $\mathrm{Ob_{\textbf{Int}}} := \mathbb{R}_{\geq 0}$
\item $\mathrm{Hom}_{\mathrm{\textbf{Int}}}(a,b) := [0,b-a]$ for all $a,b\in\mathbb{R}_{\geq 0}$, where $[0,-x] := \emptyset$ for all $x\in\mathbb{R}_{\geq 0}$,
\item $x\circ y := x+y$ for all $x\in\mathrm{Hom}_{\mathrm{\textbf{Int}}}(b,c)$ and $y\in\mathrm{Hom}_{\mathrm{\textbf{Int}}}(a,b)$, $0\leq a\leq b\leq c$.
\end{itemize}
A \textit{continuous interval sheaf} $(\mathcal{F}(a),\varphi_{a,b,x})_{0\leq b\leq a,0\leq x\leq a-b}$ is a presheaf over \textbf{Int} with values in \textbf{Set} so that, for all $0\leq\tau\leq t$ and $x,y\in\mathcal{F}(t)$,
\begin{align*}
\varphi_{t,\tau,0}(x) = \varphi_{t,\tau,0}(x)~\wedge~\varphi_{t,t-\tau,\tau}(x) = \varphi_{t,t-\tau,\tau}(y)
\end{align*}
implies that $x = y$ and, for all $x_1\in\mathcal{F}(\tau)$ and $x_2\in\mathcal{F}(t-\tau)$ with
\begin{align*}
\varphi_{t-\tau,0,0}(x_2) = \varphi_{\tau,0,\tau}(x_1)
\end{align*}
there exists $x_0\in\mathcal{F}_t$ so that
\begin{align*}
x_1 = \varphi_{t,\tau,0}(x_0)~\wedge~\varphi_{t,t-\tau,\tau}(x_0).
\end{align*}
\end{Definition}

We illustrate this definition with the realisation of differential equations as continuous interval sheaves. This realisation is ever so slightly different than the one given in~\cite[Proposition 5.1.2]{SchuSpivVasi20} where the realisation of a time-varying differential equation as time-invariant differential equation was used.

\begin{Example}\label{ex:ODE}
Let $f:\mathbb{R}^{1+n}\to\mathbb{R}^n$ be a continuous function and define, for each $\tau\in\mathbb{R}_{\geq 0}$, $\mathfrak{B}_f(\tau)$ as the set
\begin{align*}
\set{(x,\vartheta)\in\mathcal{C}^1([0,\tau],\mathbb{R}^n)\times\mathbb{R}\,\left\vert\,\begin{array}{l}\tfrac{\mathrm{d}}{\mathrm{d}t}x(t) = f(t-\vartheta,x(t))\\\text{for all}~t\in[0,\tau]\end{array}\right.};
\end{align*}
$\mathcal{C}^1([0,\tau],\mathbb{R}^n) := \set{x\vert_{[0,\tau]}\,\big\vert\,x\in\mathcal{C}^1((-\varepsilon,\tau+\varepsilon),\mathbb{R}^n), \varepsilon>0}$. Define further, for $0\leq t_0\leq t_1$ and $0\leq \tau\leq t_1-t_0$,
\begin{align*}
\varphi_{t_1,t_0,\tau}: \mathfrak{B}_f(t_1)\to\mathfrak{B}_f(t_0),\ (x,\vartheta)\mapsto (t\mapsto x(t+\tau),\vartheta-\tau);
\end{align*}
using the definition of $\mathfrak{B}_f(t_0)$, it is evident that $\varphi_{t_1,t_0,\tau}$ is well-defined. Then, the first property of a sheaf is evident, and, for all $0\leq\tau\leq t$, $(x_1,\vartheta_1)\in\mathfrak{B}_f(\tau)$ and $(x_2,\vartheta_2)\in\mathfrak{B}_f(t-\tau)$ fulfil $\varphi_{t-\tau,0,0}(x_2,\vartheta_2) = \varphi_{\tau,0,\tau}(x_1,\vartheta_1)$ if, and only if, $\vartheta_2 = \vartheta_1-\tau$ and $x_2(0) = x_1(\tau)$. Define
\begin{align*}
z: [0,t]\to\mathbb{R}^n,\qquad \text{\foreignlanguage{russian}{т}}\mapsto \begin{cases}
x_1(\text{\foreignlanguage{russian}{т}}), & \text{\foreignlanguage{russian}{т}}\in[0,\tau],\\
x_2(\text{\foreignlanguage{russian}{т}}-\tau), & \text{\foreignlanguage{russian}{т}}\in[\tau,t].
\end{cases}
\end{align*}
Since $x_1(\tau) = x_2(0)$, $z$ is continuous, and since $f$ is continuous, $z$ is continuously differentiable with
\begin{align*}
\tfrac{\mathrm{d}}{\mathrm{d}t} z(\text{\foreignlanguage{russian}{т}}) = \tfrac{\mathrm{d}}{\mathrm{d}t}x(\text{\foreignlanguage{russian}{т}}) = f(\text{\foreignlanguage{russian}{т}}-\vartheta_1,x(\text{\foreignlanguage{russian}{т}}) = f(\text{\foreignlanguage{russian}{т}}-\vartheta_1,z(\text{\foreignlanguage{russian}{т}}))
%
\end{align*}
for all $\text{\foreignlanguage{russian}{т}}\in [0,\tau]$, and
\begin{align*}
\tfrac{\mathrm{d}}{\mathrm{d}t} z(\text{\foreignlanguage{russian}{т}}) = \tfrac{\mathrm{d}}{\mathrm{d}t}x_2(\text{\foreignlanguage{russian}{т}}-\tau) & = f(\text{\foreignlanguage{russian}{т}}-\tau-\vartheta_2,x_2(\text{\foreignlanguage{russian}{т}}-\tau))\\
& = f(\text{\foreignlanguage{russian}{т}}-\vartheta_1,z(\text{\foreignlanguage{russian}{т}}))
\end{align*}

for all $\text{\foreignlanguage{russian}{т}}\in[\tau,t]$. This shows that $(z,\vartheta_1)\in\mathfrak{B}_f(t)$ with $\varphi_{t,\tau,0}((z,\vartheta_1)) = (x_1,\vartheta_1)$ and $\varphi_{t,t-\tau,\tau}((z,\vartheta_1)) = (x_2,\vartheta_2)$. Therefore, $\mathfrak{B}_f$ is a continuous interval sheaf. This sheaf may be viewed as the sheaf of all solutions of the differential equation
\begin{align}\label{eq:diff_eq}
\tfrac{\mathrm{d}}{\mathrm{d}t}x = f(t,x),
\end{align}
where a solution $x:[t_0,t_1]\to\mathbb{R}^n$, $t_1\geq t_0$, is a continuously differentiable function which fulfils~\eqref{eq:diff_eq} for all $t\in [t_0,t_1]$. Such a solution corresponds then to $(x(\cdot+t_0),t_0)\in\mathfrak{B}_f(t_1-t_0)$, and conversely each $(x,\vartheta)\in\mathfrak{B}_f(\tau)$ corresponds to the solution $x(\cdot-\vartheta)\in\mathcal{C}^1([\vartheta,\vartheta+\tau])$ of~\eqref{eq:diff_eq}.
\end{Example}

We may now give a behavioural definition of a continuous dynamical system, where ``continuous'' just means continuous time.

\begin{Definition}
A continuous dynamical system is a continuous interval sheaf $\mathfrak{B}$.
\end{Definition}

\begin{Remark}
In Willems' original framework, properties of a dynamical system, e.g. linearity, can be translated in properties of the corresponding behavioural model: A system is linear if, and only if, the signal space is linear. The same description is possible in terms of continuous interval sheaves. A linear continuous dynamical system is a continuous interval sheaf $\mathfrak{B}$ with values in the category \textbf{Vect} of vector spaces.
\end{Remark}

To deal with systems with input and output, Schultz et al. introduce a continuous machine as follows.

\begin{Definition}
A continuous machine is a tuple $(\mathfrak{B},\mathfrak{E},\mathfrak{A},\varphi)$ of continuous interval sheaves $\mathfrak{B},\mathfrak{E},\mathfrak{A}$ and a sheaf morphism $\varphi:\mathfrak{B}\to\mathfrak{A}\times\mathfrak{E}$, where $(\mathfrak{A}\times\mathfrak{E})(t) = \mathfrak{A}(t)\times\mathfrak{E}(t)$ for all $t\in\mathbb{R}_{\geq 0}$
\end{Definition}

\begin{Example}\label{ex:ISO}
Let $f\in\mathcal{C}(\mathbb{R}^{1+n+m},\mathbb{R}^n)$, $g\in\mathcal{C}(\mathbb{R}^{1+n+m},\mathbb{R}^p)$ and consider the nonlinear input-state-output system
\begin{equation}\label{eq:NLC}
\begin{aligned}
\tfrac{\mathrm{d}}{\mathrm{d}t}x(t) & = f(t,x,u(t)),\\
y(t) & = g(t,x(t),u(t)),
\end{aligned}
\end{equation}
where the control $u$ is a continuous function with values in $\mathbb{R}^m$, and define $\mathfrak{B}_{f,g}(\tau)\subseteq\mathcal{C}^1([0,\tau],\mathbb{R}^n)\times\mathcal{C}([0,\tau],\mathbb{R}^m)\times\mathbb{R}$, for all $\tau\in\mathbb{R}_{\geq 0}$, as the set
\begin{align*}
\set{(x,u,\vartheta)\,\big\vert\,\forall t\in[0,\tau]: \tfrac{\mathrm{d}}{\mathrm{d}t}x(t) = f(t-\tau,x(t),u(t))},
\end{align*}
and put, similiarly to Example~\ref{ex:ODE} for all $0\leq t_0\leq t_1$ and $0\leq \tau\leq t_1-t_0$,
\begin{align*}
\varphi_{t_1,t_0,\tau}^{\mathfrak{B}}: \mathfrak{B}_{f,g}(t_1)& \to\mathfrak{B}_{f,g}(t_0),\\
(x,u,\vartheta) & \mapsto (t\mapsto x(t+\tau),t\mapsto u(t+\tau),\vartheta-\tau).
\end{align*}
Analogously, we define $\mathfrak{A}(\tau) := \mathcal{C}([0,\tau],\mathbb{R}^m)\times\mathbb{R}$ and $\mathfrak{E} := \mathcal{C}([0,\tau],\mathbb{R}^p)\times\mathbb{R}$ with restriction morphisms
\begin{align*}
\varphi_{t_1,t_0,\tau}^{\mathfrak{A}}(u,\vartheta) & := (t\mapsto u(t+\tau),\vartheta-\tau),\\
\varphi_{t_1,t_0,\tau}^{\mathfrak{E}}(y,\vartheta) & := (t\mapsto y(t+\tau),\vartheta-\tau).
\end{align*}
Analogously to Example~\ref{ex:ODE}, it is straightforward to verify that $\mathfrak{E},\mathfrak{A},\mathfrak{B}$ are continuous interval sheaves. Define the function
\begin{align*}
\varphi:\mathfrak{B}\to\mathfrak{A}\times\mathfrak{E},\qquad (x,u,\vartheta)\mapsto \big((g(\cdot-\vartheta,x,u),\vartheta),(u,\vartheta)\big).
\end{align*}
It is evident that $\varphi$ is a morphism of continuous interval sheaves, so that $(\mathfrak{B},\mathfrak{E},\mathfrak{A},\varphi)$ is a continuous interval machine. This machine is equivalent to the control system~\eqref{eq:NLC} in the sense that there is the one-to-one correspondence between solutions of~\eqref{eq:NLC}, i.e. curves $(x,u,y)\in\mathcal{C}^1([t_0,t_1],\mathbb{R}^n)\times\mathcal{C}([t_0,t_1],\mathbb{R}^m)\times\mathcal{C}([t_0,t_1],\mathbb{R}^p)$ which fulfil~\eqref{eq:NLC} pointwise, and elements $(x(\cdot+t_0),u(\cdot+t_0),t_0)\in\mathfrak{B}_{f,g}(t_1-t_0)$. 
\end{Example}

The morphism $\varphi_\mathfrak{E}$ may be viewed as the ``input map'' and $\varphi_\mathfrak{A}$ as the ``output map'' of a given continuous machine. Note, however, that there is no explicit notion of a ``state'' of a continuous interval sheaf which differs from ``input'' and ``output'', nor is it required that the input and output differ at all. The authors of~\cite{SchuSpivVasi20} argue additionally that the notions of ``input'' and output are completely arbitrary and symmetric so that there is no fixed notion of input and output of a continuous machine. More precisely, one might think of the morphism $\varphi$ as an arbitrary ``reading'' on the system, whose results may, or may not, be used to manipulate the system by choosing convenient subsystems. Therefore, it might be argued that the theory of continuous machine is in the spirit of the behavioural theory of Willems.

\section{Port-based modelling}

\subsection{Motivation}

The idea is very simple indeed and may be illustrated with the following example. Consider a simple planar mass-spring system consisting of a single ideal spring with spring constant $k\in\mathbb{R}_{>0}$, which may be contracted and stretched along a single axis, and a single point mass $m>0$ which is attached to the spring. This system is assumed to be in a vacuum with no external forces. Then, the standard Hamiltonian model using Hooke's Law is
\begin{align*}
\tfrac{\mathrm{d}}{\mathrm{d}t}\begin{pmatrix}
x\\p
\end{pmatrix} & = \begin{bmatrix}
0 & 1\\
-1 & 0
\end{bmatrix}\begin{bmatrix}
k & 0\\
0 & \frac{1}{m}
\end{bmatrix}\begin{pmatrix}
x\\p
\end{pmatrix}.
\end{align*}
The most simple way of interaction with this system is to impose a force on the point mass. This may be viewed as adding a second spring ``on the opposite side'' of the point mass with time-varying spring constant $\kappa$, which is allowed to take arbitrary real values; since this is a thought-experiment, we do not care about the physical existence of such a spring. This configuration is sketched in Figure~\ref{fig:1}.
\begin{figure}
\centering
\begin{tikzpicture}
\def\m{3}
\def\d{0.3}
\def\h{0.5}
\def\w{1}
\def\W{0.2}
\def\H{0.7}
\draw[spring] (\m,\h/2)--++ (-\m,0) node[midway,below = 4]{$k$};
\draw[mass] (\m,0) rectangle++ (\w,\h) node[midway] {$m$};
\draw[spring,segment length = 7] (\m+\w,\h/2)--++ (\m+\d,0)node[midway,below = 4]{$\kappa(t)$};
\draw (0,-\H+\h/2) --++ (0,2*\H+\h);
\draw (2*\m+\w+\d,-\H+\h/2) --++ (0,2*\H+\h);
\draw[ground] (0,-\H+\h/2) rectangle++ (-\W,2*\H+\h);
\draw[ground] (2*\m+\w+\d,-\H+\h/2) rectangle++ (\W,2*\H+\h);
\draw[<->] (\m+\w,1.5*\h) --++ (\m+\d,0) node[midway,above = 2]{$\upsilon(t)$};
\draw[<->] (0,1.5*\h) --++ (\m,0) node[midway,above = 2]{$x(t)$};
\end{tikzpicture}
\caption{extended mass-spring-system}
\label{fig:1}
\end{figure}
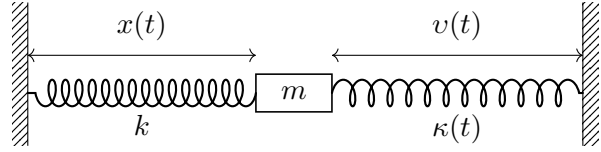
Then, we get the family of closed models for the extended mass-spring system
\begin{align*}
\tfrac{\mathrm{d}}{\mathrm{d}t}\begin{pmatrix}
x\\p\\\upsilon
\end{pmatrix} & = \begin{bmatrix}
0 & 1 & 0\\
-1 & 0 & 1\\
0 & -1 & 0
\end{bmatrix}\begin{bmatrix}
k & 0 & 0\\
0 & \frac{1}{m} & 0\\
0 & 0 & \kappa(t)
\end{bmatrix}\begin{pmatrix}
x\\p\\\upsilon
\end{pmatrix}
\end{align*}
where $\upsilon$ denotes the ``elongation'' of our second spring. When we then add the ``readings'' $u := \kappa v$ and $y := -\tfrac{\mathrm{d}}{\mathrm{d}t}\upsilon$, then we see that the port-Hamiltonian system
\begin{align*}
\tfrac{\mathrm{d}}{\mathrm{d}t}\begin{pmatrix}
x\\p
\end{pmatrix} & = \begin{bmatrix}
0 & 1\\
-1 & 0
\end{bmatrix}\begin{bmatrix}
k & 0\\
0 & \frac{1}{m}
\end{bmatrix}\begin{pmatrix}
x\\p
\end{pmatrix} + \begin{bmatrix}
0\\1
\end{bmatrix}u\\
y & = \begin{bmatrix}
0\\1
\end{bmatrix}^\top\begin{bmatrix}
k & 0\\
0 & \frac{1}{m}
\end{bmatrix}\begin{pmatrix}
x\\p
\end{pmatrix}
\end{align*}
emerges from our model systems. We may alternatively identify trajectories whose curves $x,p$ coincide of model systems with parameters $\kappa,\upsilon$ and $\kappa',\upsilon'$ whenever $\kappa \upsilon = \kappa'\upsilon'$. This identification gives a subsystem, from which also the port-Hamiltonian system after definition of $u,y$ emerges.

\subsection{Formalisation}

First, we need to establish a notion of a subsystem of a given machine. It is intuitive to say that a system contains a second system as a subsystem if the behaviour of the latter can be embedded into the behaviour of the former. In the language of Category theory, embeddings are particular morphisms, monomorphism. Morphisms of machines are straightforwardly defined using morphisms of sheaves as follows.

\begin{Definition}
Let $\mathcal{M} = (\mathfrak{B},\mathfrak{E},\mathfrak{A},\varphi)$, $\mathcal{M}' = (\mathfrak{B}',\mathfrak{E}',\mathfrak{A}',\varphi')$ be continuous machines. A triple $\Phi = (\beta,\eta,\alpha)$ is a morphism of continuous machines if either of the diagrammata
\begin{center}
\begin{tikzcd}
\mathfrak{B}\arrow{rrr}{\beta}\arrow{dr}{\varphi_{\mathfrak{E}}}\arrow{dd}{\varphi_{\mathfrak{A}}} & & & \mathfrak{B}'\arrow{dl}{\varphi_{\mathfrak{E}'}}\arrow{dd}{\varphi_{\mathfrak{A}'}}\\
& \mathfrak{E}\arrow{r}{\eta} & \mathfrak{E}'\\
\mathfrak{A}\arrow{rrr}{\alpha} & & & \mathfrak{A}'
\end{tikzcd}\qquad
\begin{tikzcd}
\mathfrak{B}\arrow{rrr}{\beta}\arrow{dr}{\varphi_{\mathfrak{E}}}\arrow{dd}{\varphi_{\mathfrak{A}}} & & & \mathfrak{B}'\arrow{dl}{\varphi_{\mathfrak{E}'}}\arrow{dd}{\varphi_{\mathfrak{A}'}}\\
& \mathfrak{E}\arrow{r}{\eta} & \mathfrak{A}'\\
\mathfrak{A}\arrow{rrr}{\alpha} & & & \mathfrak{E}'
\end{tikzcd}
\end{center}
whose arrows are morphisms of \textbf{Int}-sheaves, is commutative; we write $\Phi:\mathcal{M}\to\mathcal{M}'$.
\end{Definition}

\begin{Remark}
The two possibilities of machine morphisms as families of morphisms $(\mathfrak{B}\to\mathfrak{B}',\mathfrak{E}\to\mathfrak{E}',\mathfrak{A}\to\mathfrak{A}')$ and $(\mathfrak{B}\to\mathfrak{B}',\mathfrak{E}\to\mathfrak{A}',\mathfrak{A}\to\mathfrak{E}')$ reflect the ambiguity of the choice of the sheafes $\mathfrak{A}$ and $\mathfrak{E}$.  For example, an input-state-output system~\eqref{eq:NLC} has the realisation as continuous machine of Example~\ref{ex:ISO}, where the projection onto $\mathfrak{A}$ selects the input and the projection onto $\mathfrak{E}$ selects the output of a trajectory; $\varphi_\mathfrak{A}$ may be interpreted as the input map and $\varphi_\mathfrak{E}$ as the output map. The exchange of input and output map may be interpreted as exchanging the problem of determining an output for an exlpicit differential equation, given by an input function, to the inverse control problem of determining an input which produces a given output. The formal equivalence of these problems is reflected in the the isomorphy of the machines $(\mathfrak{B},\mathfrak{E},\mathfrak{A},\varphi_\mathfrak{E}\oplus\varphi_\mathfrak{A})$ and $(\mathfrak{B},\mathfrak{A},\mathfrak{E},\varphi_\mathfrak{A}\oplus\varphi_\mathfrak{E})$.
\end{Remark}

Now, the notion of a subsystem reveals itself.

\begin{Definition}
Let $\mathcal{M} = (\mathfrak{B},\mathfrak{E},\mathfrak{A},\varphi)$, $\mathcal{M}' = (\mathfrak{B}',\mathfrak{E}',\mathfrak{A}',\varphi')$ be continuous machines. $\mathcal{M}$ is a \textit{subsystem} of $\mathcal{M}'$ if, and only if, there exists a monomorphism $\Phi: \mathcal{M}\hookrightarrow\mathcal{M}'$.
\end{Definition}

Lastly, we need a notion of a closed system. Intuitively, a system is closed, when there is no interaction between the system with any surroundings; in particular, there are no ``readings'', or ``measurements'' on the system. Recalling the interpretation of the morphism $\mathfrak{B}\to\mathfrak{A}\times\mathfrak{E}$ of a continuous machine as ``measurements'', the following definition presents itself.

\begin{Lemma}
Each constant presheaf over $\mathrm{\textbf{Int}}$ with values in \textbf{Set} is a continuous interval sheaf. In particular, all constant continuous interval sheaves are isomorphic.
\end{Lemma}
\begin{proof}
Let $x$ be a set with identity morphism $1_x$. Put $\mathcal{X}\in\mathrm{Fct}(\mathrm{\textbf{Int}}^\text{op},\mathcal{C})$ as the constant functor that assigns $x$ to each interval object and $1_x$ to each interval morphism. To see that this is indeed a sheaf, observe that each $t\in\mathrm{Ob_\textbf{Int}} = \mathbb{R}_{\geq 0}$ admits at least one nonempty covering, namely $(0,t)$, and hence the first sheaf property is fulfilled; the second property is obviously true. Thus, $\mathcal{X}$ is a continuous interval sheaf. Per definitionem, a constant sheaf is a sheaf that is isomorphic to the sheafification of a constant presheaf. Since constant presheaves are sheaves, each constant sheaf is isomorphic to a constant presheaf; and constant presheaves are evidently isomorphic. Therefore, all constant sheaves are indeed isomorphic.  
\end{proof}

Now, we may describe a port-control structure for a closed system, which has certain desired properties, as follows.

\begin{Definition}\label{def:prt-control}
Let $\mathfrak{B}_0$ and $\mathfrak{B}$ be continuous interval sheaves and let $A:\mathfrak{B}_0\hookrightarrow\mathfrak{B}$ be a monomorphism of sheaves so. Choose continuous interval sheaves $\mathfrak{A},\mathfrak{E}'$ and a morphism $\varphi:\mathfrak{B}\to\mathfrak{A}\times\mathfrak{E}$. A continuous machine $(\mathfrak{B}',\mathfrak{A}',\mathfrak{E}',\varphi')$ is a port-controlled $\mathfrak{B}_0$-system with $(\mathfrak{B},\mathfrak{A},\mathfrak{E},\varphi)$-interface if there exist a constant continuous interval sheaf $\mathfrak{O}$, a continuous interval sheaf $\mathfrak{A}_0$ and monomorphisms $(A,\Phi_1,\Phi_2),\Psi,\Xi$ of continuous machines so that the following diagram is commutative:
\begin{center}
\begin{tikzcd}
(\mathfrak{B}_0,\mathfrak{O},\mathfrak{A}_0,\varphi)\arrow[hook]{rr}{(A,\Phi_1,\Phi_2)}\arrow[hook]{dr}{\Psi} & &  (\mathfrak{B},\mathfrak{A},\mathfrak{E},\varphi)\\
& (\mathfrak{B}',\mathfrak{A}',\mathfrak{E}',\varphi')\arrow[hook]{ur}{\Xi} &
\end{tikzcd}
\end{center}
\end{Definition}

\begin{Remark}
The choice of $\mathfrak{A}$ and $\mathfrak{E}$ and the projection $\varphi$ may be interpreted as the (arbitrary) choice of port-variables in the originally closed system $\mathfrak{B}$. The original system, to which we strive to assign a port-structure, is a closed subsystem of $\mathfrak{B}$. The choice of the subsystem $(\mathfrak{B}',\mathfrak{A}',\mathfrak{E}',\varphi')$ may be viewed as eliminating certain more or less undesired configurations of port-variables or unnecessary ``states'' of the enclosing model $\mathfrak{B}$. An example would be our motivating mass spring system, and we include in the model system a totally unrelated system, e.g. a blinking lamp, which is not needed nor desired for our particular choice of port-Hamiltonian system. Eliminating this system from the model system is a specific choice of subsystem $(\mathfrak{B}',\mathfrak{A}',\mathfrak{E}',\varphi')$.
\end{Remark}

\section{Applications}

\subsection{Port-Hamiltonian systems on Euclidean space}

Port-Hamiltonian systems on the Euclidean space are systems which have the form
\begin{equation}\label{eq:pH}
\begin{aligned}
\tfrac{\mathrm{d}}{\mathrm{d}t}x & = (J(x)-R(x))\nabla H(x) + B(x)u\\
y & = B(x)^\top \nabla H(x),
\end{aligned}
\end{equation}
where $J,R\in\mathcal{C}(\mathbb{R}^n,\mathbb{R}^{n\times n})$ with $J(\cdot) = -J(\cdot)^\top$ and $R(\cdot) = R(\cdot)^\top\geq 0$, and $H\in\mathcal{C}^1(\mathbb{R}^n,\mathbb{R})$. These systems may be viewed as a particular choice of port-structure for a dissipative Hamiltonian system modelled by
\begin{align*}
\tfrac{\mathrm{d}}{\mathrm{d}t}x = (J(x)-R(x))\nabla H(x).
\end{align*}
Put as $\mathfrak{B}_0$ the behaviour of Example~\ref{ex:ODE} with $f = (J-R)\nabla H$, i.e. $\mathfrak{B}_0(\tau)\subseteq\mathcal{C}^1([0,\tau],\mathbb{R}^n)\times\mathbb{R}$ is the set
\begin{align*}
\set{(x,\vartheta)\,\big\vert\,\forall t\in [0,\tau]: \tfrac{\mathrm{d}}{\mathrm{d}t}x(t) = (J(x(t))-R(x(t)))\nabla H(x(t))}.
\end{align*}
Note that, due to time-invariance of the system, the time-shift parameter $\vartheta$ may be omitted; this would, however, complicate the following construction. Define the continuous matrix fields
\begin{align*}
\mathcal{J}:\mathbb{R}^{n+m} & \to\mathbb{R}^{(n+m)\times(n+m)},\qquad (x,\zeta)\mapsto \begin{bmatrix}
J(x) & B(x)\\
-B(x)^\top & 0
\end{bmatrix},\\
\mathcal{R}:\mathbb{R}^{n+m} & \to\mathbb{R}^{(n+m)\times(n+m)},\qquad (x,\zeta)\mapsto \begin{bmatrix}
R(x) & 0\\
0 & 0
\end{bmatrix}
\end{align*}
and put as $\mathfrak{B}$ the behaviour of the collection of dissipative Hamiltonian systems with structure matrices $\mathcal{J}$ and $\mathcal{R}$ whose Hamiltonian function may be decomposed into the direct sum of $H$ and an auxiliary, possibly time-varying, Hamiltonian function $H_\alpha\in\mathcal{C}^{0,1}([0,\tau]\times\mathbb{R}^{m},\mathbb{R}))$, i.e. $H_\alpha$ is continuous in the first and continuously differentiable in the last $m$ entries. This yields $\mathfrak{B}(\tau)\subseteq \mathcal{C}^1([0,\tau],\mathbb{R}^n)\times\mathbb{R}\times\mathcal{C}^{0,1}([0,\tau]\times\mathbb{R}^m,\mathbb{R})$ as the set
\begin{align*}
\set{(\xi,\vartheta,H_\alpha)\left\vert\!\begin{array}{l}\begin{array}{cl}\tfrac{\mathrm{d}}{\mathrm{d}t}\xi(t) = & \mathcal{J}(\xi(t))\nabla (H\oplus H_\alpha(t-\vartheta,\cdot))(\xi(t))\\ & -\mathcal{R}(\xi(t))\nabla (H\oplus H_\alpha(t-\vartheta,\cdot))(\xi(t))\end{array}\\\text{for all}~t\in [0,\tau]\end{array}\right.\!\!}
\end{align*}
With the restriction morphisms
\begin{align*}
\varphi^{\mathfrak{B}}_{t_1,t_0,\tau}(\xi,\vartheta,H_\alpha) := (t\mapsto \xi(t+\tau),\vartheta-\tau,H_\alpha(\cdot+\tau,\cdot))
\end{align*}
for all $0\leq t_0\leq t_1$ and $0\leq\tau\leq t_1-t_0$, $\mathfrak{B}$ becomes a continuous interval sheaf which admits the embedding $A:\mathfrak{B}_0\hookrightarrow\mathfrak{B}$
\begin{align*}
A(x,\vartheta):=\left(\left(x,-\int_0^\cdot B(x(w))^\top\nabla H(x(w))\,\mathrm{d}w\right),\vartheta,0\right).
\end{align*}
A direct calculation yields for all $(x,\vartheta)\in\mathfrak{B}_0(\tau)$ and $\zeta := -\int_0^\cdot B(x(w))^\top\nabla H(x(w))\in\mathcal{C}^1([0,\tau],\mathbb{R}^m)$ and for all $t\in[0,\tau]$,
\begin{align*}
& \tfrac{\mathrm{d}}{\mathrm{d}t}\begin{pmatrix}x\\-\int_0^\cdot B(x(w))^\top\nabla H(x(w))\,\mathrm{d}w\end{pmatrix}(t)\\
&\quad = \begin{pmatrix}
(J(x(t))-R(x(t)))\nabla H(x(t))\\
-B(x(t))^\top\nabla H(x(t))
\end{pmatrix}\\
&\quad = \left(\mathcal{J}-\mathcal{R}\right)(x(t),\zeta(t))\nabla (H\oplus 0(t-\vartheta,\cdot))(x(t),\zeta(t))
\end{align*}
and hence $A$ is indeed well-defined; $A$ is a monomorphism since each $A_\tau:\mathfrak{B}_0(\tau)\to\mathfrak{B}(\tau)$ is by construction injective. The images of the morphisms of sheaves
\begin{align*}
& \varphi_\mathfrak{A}:\mathfrak{B}\to\mathcal{C}(\cdot,\mathbb{R}^m)\times\mathbb{R},\\
&\qquad (\xi,\vartheta,H_\alpha) \mapsto (\mathrm{d}p_2(\nabla(H\oplus H_\alpha))\circ\xi,\vartheta),\\
& \varphi_\mathfrak{E}:\mathfrak{B}\to\mathcal{C}(\cdot,\mathbb{R}^m)\times\mathbb{R},\ (\xi,\vartheta,H_\alpha)  \mapsto \left(-\tfrac{\mathrm{d}}{\mathrm{d}t}p_2\circ\xi,\vartheta\right),
\end{align*}
where $p_2:\mathbb{R}^{n+m}\to\mathbb{R}^m$ is the Cartesian projection onto the second component bear the structure of continuous interval sheaves $\mathfrak{A}$ and $\mathfrak{B}$. Then, we have a continuous machine $(\mathfrak{B},\mathfrak{A},\mathfrak{E},(\varphi_\mathfrak{A},\varphi_\mathfrak{B}))$. Define $\varphi_{\mathfrak{A}_0}:\mathfrak{B}_0\to \mathcal{C}(\cdot,\mathbb{R}^m)\times\mathbb{R}$ as the pullback of $\varphi_\mathfrak{A}$ under $A$, i.e. $\varphi_{\mathfrak{A}_0} := \varphi_{\mathfrak{A}}\circ A$, and let $\mathfrak{A}_0$ the image of $\varphi_{\mathfrak{A}_0}$, which is a sheaf. Let further $\mathfrak{O}$ be the continuous interval sheaf that assigns each interval the set $\set{0}$. By construction of $A$, we see that the projection $\varphi_{\mathfrak{O}}:\mathfrak{B}_0\to\mathfrak{O}$ is exactly $\varphi_{\mathfrak{E}}\circ A$ and hence $(A,\mathrm{id}_{\mathfrak{A}_0},\mathrm{id}_{\mathfrak{E}_0})$, where $\mathrm{id}_{\mathfrak{A}_0}:\mathfrak{A}_0\hookrightarrow\mathfrak{A}$ and $\mathrm{id}_{\mathfrak{E}_0}:\mathfrak{E}_0\hookrightarrow\mathfrak{E}$ denote the canonical embeddings, is a monomorphism of continuous machines. Recall from Example~\ref{ex:ISO} that a representation of the input-state-output system~\eqref{eq:pH} as a continuous machine is $(\widehat{\mathfrak{B}},\widehat{\mathfrak{A}},\widehat{\mathfrak{E}},(\varphi_{\widehat{\mathfrak{A}}},\varphi_{\widehat{\mathfrak{E}}}))$ with $\widehat{\mathfrak{B}}(\tau)\subseteq \mathcal{C}^1([0,\tau],\mathbb{R}^n)\times\mathcal{C}([0,\tau],\mathbb{R}^m)\times\mathbb{R}$ given as the set
\begin{align*}
\set{(x,u,\vartheta)\,\left\vert\,\forall t\in[0,\tau]:\begin{array}{cl}\tfrac{\mathrm{d}}{\mathrm{d}t}x(t) =  & J(x(t))\nabla H(x(t))\\ & -R(x(t))\nabla H(x(t))\\ & +B(x(t))u(t)\end{array}\right.}
\end{align*}
and $\widehat{\mathfrak{A}}$ and $\widehat{\mathfrak{E}}$ given as the images of the projections
\begin{align*}
\varphi_{\widehat{\mathfrak{A}}}:\widehat{\mathfrak{B}} & \to\mathcal{C}(\cdot,\mathbb{R}^m)\times\mathbb{R},\quad (x,u,\vartheta)\mapsto (u,\vartheta),\\
\varphi_{\widehat{\mathfrak{E}}}:\widehat{\mathfrak{B}} & \to\mathcal{C}(\cdot,\mathbb{R}^m)\times\mathbb{R},\quad (x,u,\vartheta)\mapsto (B(x(t))^\top\nabla H(x(t)),\vartheta).
\end{align*}
We show that $(\widehat{\mathfrak{B}},\widehat{\mathfrak{A}},\widehat{\mathfrak{E}},(\varphi_{\widehat{\mathfrak{A}}},\varphi_{\widehat{\mathfrak{E}}}))$ fits into the diagram
\begin{center}
\begin{tikzcd}[ampersand replacement=\&]
(\mathfrak{B}_0,\mathfrak{A}_0,\mathfrak{O},\varphi)\arrow[hook]{rr}{(A,\mathrm{id}_{\mathfrak{A}_0},\mathrm{id}_{\mathfrak{E}_0})}\arrow[hook]{dr}{\Psi} \& \&  (\mathfrak{B},\mathfrak{A},\mathfrak{E},\xi)\\
\& (\widehat{\mathfrak{B}},\widehat{\mathfrak{A}},\widehat{\mathfrak{E}},\psi)\arrow[hook]{ur}{\Xi} \&
\end{tikzcd}
\end{center}
where $\varphi := \varphi_{\mathfrak{A}_0}\oplus\varphi_\mathfrak{O}$, $\psi := \varphi_{\widehat{\mathfrak{A}}}\oplus\varphi_{\widehat{\mathfrak{E}}}$ and $\xi := \varphi_\mathfrak{A}\oplus\varphi_\mathfrak{E}$
To this end, define the embeddings $\Psi = (A,\mathrm{id}_{\mathfrak{A}_0},\mathrm{id}_{\mathfrak{E}_0})$ and $\Xi = (\Xi_0,\Xi_1,\Xi_2)$ as
\begin{align*}
{\small \Xi_0(x,u,\vartheta) := \left(\left(x,-\int_0^\cdot B(x(w))^\top\nabla H(x(w))\,\mathrm{d}w\right)\vartheta, u(\cdot)^\top\right)}
\end{align*}
and $\Xi_1$ and $\Xi_2$ the canonical embedding of $\widehat{\mathfrak{A}}\subseteq\mathfrak{A}$ and $\widehat{\mathfrak{A}}\subseteq\mathfrak{A}$, respectively. This shows that $(\widehat{\mathfrak{B}},\widehat{\mathfrak{A}},\widehat{\mathfrak{E}},(\varphi_{\widehat{\mathfrak{A}}},\varphi_{\widehat{\mathfrak{E}}}))$ is a port-controlled dissipative Hamiltonian system with time-varying Hamiltonian interface.

\subsection{Metriplectic systems}

Metriplectic systems are an extension of Hamiltonian systems to incorporate dissipative effects due to irreversible generation of entropy. They are used as models for systems in fluid dynamics and thermodynamics, see e.g.~\cite{Morr84}. In their most simple, finite-dimensional incarnation, metriplectic systems are given as follows.

\begin{Definition}
The dynamical system $M(J,G,H,S)$ given by the differential equation
\begin{align}\label{eq:metriplectic}
\tfrac{\mathrm{d}}{\mathrm{d}t} x = J(x)\nabla H(x) + G(x)\nabla S(x),
\end{align}
where $J,G\in\mathcal{C}(\mathbb{R}^n,\mathbb{R}^{n\times n})$ fulfil $J(x) = -J(x)^\top$ and $G(x) = G(x)^\top\geq 0$ for all $x\in\mathbb{R}^n$, and $H,S\in\mathcal{C}^1(\mathbb{R}^n,\mathbb{R})$ are subject to the noninteraction condition
\begin{align}\label{eq:noninteraction}
\forall x\in\mathbb{R}^n: J(x)\nabla S(x) = G(x)\nabla H(x) = 0.
\end{align}
\end{Definition}

Analogously to Hamiltonian systems, we propose a straightforward definition of port-controlled metriplectic systems using the behavioural framework. A representation of the behaviour of the metriplectic system~\eqref{eq:metriplectic} $\mathfrak{B}_0$ so that $\mathfrak{B}_0(\tau)\subseteq\mathcal{C}^1([0,\tau],\mathbb{R}^n)\times\mathbb{R}$ is
\begin{align*}
\set{(x,\vartheta)\,\big\vert\,\forall t: \tfrac{\mathrm{d}}{\mathrm{d}t}x(t) = J(x)\nabla H(x(t)) + G(x(t))\nabla S(x(t))}.
\end{align*}
Let $\widetilde{J},\widetilde{G}\in\mathcal{C}(\mathbb{R}^n,\mathbb{R}^{m\times m})$ and $B,A\in\mathcal{C}(\mathbb{R}^n,\mathbb{R}^{n\times m}$, and define 
\begin{align*}
\mathcal{J}:\mathbb{R}^{n+m} & \to\mathbb{R}^{(n+m)\times(n+m)},\qquad (x,\zeta)\mapsto \begin{bmatrix}
J(x) & B(x)\\
-B(x)^\top & \widetilde{J}(x)
\end{bmatrix},\\
\mathcal{G}:\mathbb{R}^{n+m} & \to\mathbb{R}^{(n+m)\times(n+m)},\qquad (x,\zeta)\mapsto \begin{bmatrix}
G(x) & A(x)\\
A(x)^\top & \widetilde{G}(x)
\end{bmatrix}.
\end{align*}
Additionally, it is assumed that $\widetilde{J},\widetilde{G},A$ and $B$ are chosen so that $\mathcal{J}$ is pointwise antisymmetric and $\mathcal{G}$ is pointwise symmetric and positive semidefinite. We define $\mathfrak{B}$ as the behaviour of all metriplectic systems with structure matrices $\mathcal{J}$ and $\mathcal{R}$ whose Hamiltonian and entropy functions admit a decomposition into the direct sum of $H$ and a time-varying auxiliary Hamiltonian function $H_\alpha\in\mathcal{C}^{0,1}([0,\tau]\times\mathbb{R}^{m},\mathbb{R}))$, and $S$ and $S_\alpha\in\mathcal{C}^{0,1}([0,\tau]\times\mathbb{R}^{m},\mathbb{R}))$, respectively. We put, for each $\tau\in\mathbb{R}_{\geq 0}$, $\mathfrak{B}(\tau)\subseteq \mathcal{C}^1([0,\tau],\mathbb{R}^n)\times\mathbb{R}\times\mathcal{C}^{0,1}([0,\tau]\times\mathbb{R}^m,\mathbb{R})\times\mathcal{C}^{0,1}([0,\tau]\times\mathbb{R}^m,\mathbb{R})$ as the set of all $(\xi,\vartheta,H_\alpha,S_\alpha)$ with
\begin{align*}
\forall t\in [0,\tau]: \begin{array}{cl}\tfrac{\mathrm{d}}{\mathrm{d}t}\xi(t) = & \mathcal{J}(\xi(t))\nabla (H\oplus H_\alpha(t-\vartheta,\cdot))(\xi(t))\\
& +\mathcal{G}(\xi(t))\nabla (S+S_\alpha(t-\vartheta,\cdot))(\xi(t)),
\end{array}
\end{align*}
and $\widetilde{J}(\cdot)\nabla(H\oplus H_\alpha)(\cdot) = \widetilde{G}(\cdot)\nabla(S\oplus S_\alpha)(\cdot)\equiv 0$,
and define the canonical restriction morphisms so that $\mathfrak{B}$ is a continuous interval sheaf. In particular, $\mathfrak{B}$ admits the embedding $A:\mathfrak{B}_0\hookrightarrow\mathfrak{B}$
\begin{align*}
A(x,\vartheta) & :=\Big(\Big(x,-\int_0^\cdot B(x(w))^\top\nabla H(x(w))\\
& \qquad+A(x(w))^\top\nabla S(x(w))\,\mathrm{d}w\Bigg),\vartheta,0,0\Big);
\end{align*}
a direct calculation yields that $A$ is indeed a well-defined monomorphism. Define the projections\vspace{-5mm}
\begin{center}
\scalebox{0.9}{\parbox{\linewidth}{\begin{align*}
& \varphi_\mathfrak{A}(\xi,\vartheta,H_\alpha,S_\alpha) :=  (\mathrm{d}p_2(\nabla(H\oplus H_\alpha))\circ\xi,\mathrm{d}p_2(\nabla(S\oplus S_\alpha))\circ\xi.\vartheta),\\
& \varphi_\mathfrak{E}(\xi,\vartheta,H_\alpha,S_\alpha):=\left(-\tfrac{\mathrm{d}}{\mathrm{d}t}p_2\circ\xi,\vartheta\right),
\end{align*}}}
\end{center}\vspace{-1mm}
and equip their images $\mathfrak{A}$ and $\mathfrak{B}$ with the canonical structure of a continuous interval sheaf so that $\varphi_\mathfrak{A}$ and $\varphi_\mathfrak{B}$ are sheaf-morphisms, which gives a continuous machine $(\mathfrak{B},\mathfrak{A},\mathfrak{E},(\varphi_\mathfrak{A},\varphi_\mathfrak{E}))$. Choosing the particular auxiliary Hamiltonians that are induced by continuous curves in the dual space, we get the behaviour $\widetilde{\mathfrak{B}}$, with, for each $\tau\in\mathbb{R}_{\geq 0}$, by $(x,u,\tau,\vartheta)\in \widetilde{\mathfrak{B}}(\tau)$ if, and only if,
\begin{align*}
\forall t\in[0,\tau]:\begin{array}{cl}\tfrac{\mathrm{d}}{\mathrm{d}t}x(t) = & J(x(t))\nabla H(x(t))+G(x(t))\nabla S(x(t))\\
& +B(x(t))u(t)+A(x(t))\tau(t),
\end{array}
\end{align*}
and $B(\cdot)\tau = G(\cdot)u\equiv 0$, $B(\cdot)^\top\nabla S(\cdot) - \widetilde{J}(\cdot)\tau\equiv 0$, and $A(\cdot)^\top\nabla H(\cdot)+\widetilde{G}(\cdot)u\equiv 0$
This behaviour admits an analogous embedding $\Xi_0$, which allows us to pull back the projections $\varphi_\mathfrak{A}$ and $\varphi_\mathfrak{B}$ giving a morphism of continuous machines $\Xi:(\widetilde{\mathfrak{B}},\widetilde{\mathfrak{A}},\widetilde{\mathfrak{E}},\widetilde{\varphi})\hookrightarrow (\mathfrak{B},\mathfrak{A},\mathfrak{E},(\varphi_\mathfrak{A},\varphi_\mathfrak{E}))$. Then, we can complete the diagram
\begin{center}
\begin{tikzcd}
(\mathfrak{B}_0,\mathfrak{A}_0,\mathfrak{O},\varphi)\arrow[hook]{rr}{(A,\mathrm{id}_{\mathfrak{A}_0},\mathrm{id}_{\mathfrak{E}_0})}\arrow[hook]{dr}{\Psi} & &  (\mathfrak{B},\mathfrak{A},\mathfrak{E},\xi)\\
& (\widetilde{\mathfrak{B}},\widetilde{\mathfrak{A}},\widetilde{\mathfrak{E}},\widetilde{\varphi})\arrow[hook]{ur}{\Xi} &
\end{tikzcd}
\end{center}
analogously to the Hamiltonian case. This gives us that the continuous machine $(\widetilde{\mathfrak{B}},\widetilde{\mathfrak{A}},\widetilde{\mathfrak{E}},\widetilde{\varphi})$ is a port-controlled metriplectic system with time-varying metriplectic interface. When we choose the matrix fields $A$ and $B$ so that additionally the strong noninteraction condition $B(\cdot)^\top\nabla S(\cdot) = A(\cdot)^\top\nabla H(\cdot) \equiv 0$ is fulfilled, then we get the following class of input-state-output systems.

\begin{Definition}
A port-metriplectic system is a control system
\begin{equation}\label{eq:port-GENERIC}
\begin{aligned}
\tfrac{\mathrm{d}}{\mathrm{d}t} x & = J(x)\nabla  H(x) + G(x)\nabla S(x)+B(x)u+A(x)\tau\\
y & = B(x)^\top\nabla H(x) - A(x)^\top\nabla S(x)-\widetilde{J}(x)u-\widetilde{G}(x)\tau
\end{aligned}
\end{equation}
where $J,R\in\mathcal{C}(\mathbb{R}^n,\mathbb{R}^{n\times n})$ with $J(\cdot)^\top = -J(\cdot)$ and $G(\cdot)^\top = G(\cdot)$, $A,B\in\mathcal{C}(\mathbb{R}^n,\mathbb{R}^{n\times m})$, $\widetilde{J},\widetilde{G}\in\mathcal{C}(\mathbb{R}^n,\mathbb{R}^{m\times m})$ with $\widetilde{J}(\cdot) = -\widetilde{J}(\cdot)^\top$ and
\begin{align*}
\begin{bmatrix}
G(\cdot) & A(\cdot)\\
A(\cdot)^\top & \widetilde{G}(\cdot)
\end{bmatrix}\geq 0,
\end{align*}
and $S,H\in\mathcal{C}^1(\mathbb{R}^n,\mathbb{R})$ and $u,\tau\in\mathcal{C}(\mathbb{R},\mathbb{R}^m)$ so that
\begin{align*}
J\partial S = G\partial H = B\tau = Au \equiv 0
\end{align*}
and
\begin{align*}
B(\cdot)^\top\nabla S(\cdot) = A(\cdot)^\top\nabla H(\cdot) = \widetilde{J}(\cdot)\tau = \widetilde{G}(\cdot)u\equiv 0.
\end{align*}
\end{Definition}

\section{Conclusion}

A general approach of assigning port-control structures to a general behavioural dynamical system was presented. The versatility of this approach has been demonstrated with examples.

\bibliographystyle{alpha}
\bibliography{Literatur} 
\end{document}